\documentclass[10pt]{amsart}
\usepackage{graphicx}
\usepackage{amsmath, amsthm,  amsfonts, amscd, parskip} %parskip removes automatic indent
\usepackage{url}
\usepackage[letterpaper, margin=3cm]{geometry}
\usepackage{slashed} % For Dirac operator
\usepackage{tikz-cd}
\usepackage{mathtools}
\allowdisplaybreaks % aligned equations can break over pages

\theoremstyle{plain}
\newtheorem{theorem}{Theorem}[section]
\newtheorem{lemma}[theorem]{Lemma}
\newtheorem{proposition}[theorem]{Proposition}

\newtheorem{corollary}[theorem]{Corollary}
\newtheorem*{gssconjecture}{Global Spherical Shell Conjecture}
\newtheorem*{theorem*}{Theorem}
\newtheorem*{corollary*}{Corollary}

\theoremstyle{definition}

\theoremstyle{remark}

\newtheorem{remark}{Remark}[section]

\numberwithin{equation}{section}
\numberwithin{figure}{section}

\title[Yamabe invariants of Inoue Surfaces]{The Yamabe invariants of Inoue Surfaces, Kodaira Surfaces, and their blowups}
\author{Michael Albanese}
\address{Universit\'e du Qu\'ebec \`a Montr\'eal \\ D\'epartement de math\'ematiques}
\email{michael.albanese@cirget.ca} 

\begin{document}

\maketitle

\begin{abstract} 
Shortly after the introduction of Seiberg-Witten theory, LeBrun showed that the sign of the Yamabe invariant of a compact K\"ahler surface is determined by its Kodaira dimension. In this paper, we show that LeBrun's Theorem is no longer true for non-K\"ahler surfaces. In particular, we show that the Yamabe invariants of Inoue surfaces and their blowups are all zero. We also take this opportunity to record a proof that the Yamabe invariants of Kodaira surfaces and their blowups are all zero, as previously indicated by LeBrun.
\end{abstract}

\section{Introduction}

The complex geometric properties of a compact Riemann surface are intimately related to its Riemannian geometry. By the Gauss-Bonnet theorem, the total scalar curvature of any Riemannian metric is a positive multiple of the Euler characteristic. It follows that the cases of Kodaira dimension $-\infty$, $0$, and $1$, corresponds precisely to the total scalar curvature being positive, zero, and negative.

In higher dimensions, the total scalar curvature of a Riemannian metric is no longer independent of the metric, and hence forms a functional on the space of all Riemannian metrics. After normalising the functional, one can use it to define the Yamabe invariant $Y(M)$ of a closed manifold $M$ via a min-max definition (see section 3), which coincides with the total scalar curvature when $M$ is a Riemann surface. In \cite{LeBrun2}, LeBrun showed that the relationship between the trichotomies for Riemann surfaces carries over to K\"ahler surfaces with the Yamabe invariant taking the place of total scalar curvature.

\begin{theorem*}
{\bf (LeBrun)} Let $M$ be a connected compact complex surface which admits a K\"ahler metric. Then $Y(M)$ is positive if $\kappa(M) = -\infty$, zero if $\kappa(M) = 0$ or $1$, and negative if $\kappa(M) = 2$.
\end{theorem*}

As LeBrun pointed out, there is no higher dimensional analogue of this trichotomy; in particular, manifolds of general type can have positive Yamabe invariant in higher dimensions.

Unlike the case of Riemann surfaces, there exist complex manifolds in complex dimension two which do not admit K\"ahler metrics. With this in mind, a natural question to ask is whether the K\"ahler hypothesis of the above theorem is necessary. The main result of this paper is the following:

\begin{theorem*}
Let $M$ be an Inoue surface or a blownup Inoue surface. Then $Y(M) = 0$.
\end{theorem*}

As Inoue surfaces are non-K\"ahler surfaces with Kodaira dimension $-\infty$, we immediately conclude the following:

\begin{corollary*}
The K\"ahler hypothesis of LeBrun's Theorem is necessary.
\end{corollary*}

In addition, we compute the Yamabe invariant of Kodaira surfaces and their blowups, as previously indicated by LeBrun.

{\bf Disclaimer.} The results of this paper were proved during the author's PhD and were conceived prior to some recent results which give rise to different proofs of Theorem \ref{Inouepsc} and Theorem \ref{Kodairapsc} which we outline here. Inoue surfaces and Kodaira surfaces are solvmanifolds as shown by Hasegawa \cite{Hasegawa}, so they are enlargeable in the sense of Gromov and Lawson \cite{GL1}. Using Schoen and Yau's upgraded minimal hypersurface technique \cite{SY17}, Cecchini and Schick \cite{CS} have proved that a compact enlargeable manifold cannot admit a metric of positive scalar curvature -- prior to this, one needed the additional assumption that the manifold was spin (note that blowups of Inoue surfaces and Kodaira surfaces are enlargeable, but not spin). The result of Cecchini and Schick in low dimensions (namely, at most $8$) can be deduced from Gromov and Lawson \cite[Theorem 12.1]{GL2}, at least for compactly enlargeable manifolds. None of these results are employed here, instead we use a different argument which may be of independent interest. 

{\bf Acknowledgements.} First and foremost, I would like to thank my former advisor Claude LeBrun for introducing me to this problem and guiding me during this work. I would also like to thank Blaine Lawson for answering my questions regarding his work on metrics of positive scalar curvature. The exposition has benefited from illuminating conversations with Georges Dloussky and S\"onke Rollenske, many of which took place at the Banff International Research Station for Mathematical Innovation and Discovery where I presented these results during the workshop ``Bridging the Gap between K\"ahler and non-K\"ahler Complex Geometry". Finally, thanks to the anonymous referee for their helpful suggestions.

\section{Compact Complex Surfaces}

Recall that for a connected compact complex $n$-dimensional manifold $X$, its \textit{Kodaira dimension}, denoted by $\kappa(X)$, is defined to be $-\infty$ if $h^0(X, K_X^{\ell}) = \dim H^0(X, K_X^{\ell}) = 0$ for all $\ell > 0$; here $K_X$ denotes the canonical bundle of $X$. Otherwise, the Kodaira dimension of $X$ is given by
$$\kappa(X) = \limsup \frac{\log h^0(X, K_X^{\ell})}{\log \ell}.$$
The Kodaira dimension takes values in $\{-\infty, 0, 1, \dots, n\}$, and if $X$ has maximal Kodaira dimension, i.e. $\kappa(X) = n$, we say that $X$ is of \textit{general type}.

Now suppose $n = 2$. By the Kodaira-Enriques classification, the surface $X$ has a minimal model in precisely one\footnote{$X$ need not have a unique minimal model, but any two minimal models belong to the same class.} of the following 10 classes, organised in the table below.

\begin{center}\begin{tabular}{|c|c|c}
\hline
$\kappa(X)$ & $b_1(X)$ even & \multicolumn{1}{c|}{$b_1(X)$ odd} \\ \hline
$-\infty$ & \begin{tabular}{c} Rational\\ Ruled\end{tabular} & \multicolumn{1}{c|}{Class $\mathrm{VII}_0$} \\ \hline
$0$ & \begin{tabular}{c} $K3$\\ Enriques\\ Tori\\ Hyperelliptic\end{tabular} & \multicolumn{1}{c|}{Kodaira}  \\ \hline
$1$ & Properly Elliptic  & \multicolumn{1}{c|}{Properly Elliptic}  \\ \hline
$2$ & General Type  &                        \\ \cline{1-2}
\end{tabular}\end{center}

The parity of $b_1(X)$ is of interest because it determines when $X$ admits a K\"ahler metric: a connected compact complex surface admits a K\"ahler metric if and only if $b_1(X)$ is even. This was initially proved using the Kodaira-Enriques classification with the final case of $K3$ surfaces completed by Siu \cite{Siu}, while a proof without using the classification was obtained independently by Buchdahl \cite{Buchdahl} and Lamari \cite{Lamari}.

We now turn our attention to the two classes of interest for this paper, namely class $\mathrm{VII}_0$ surfaces and Kodaira surfaces.

\subsection{Class $\mathrm{VII}_0$ Surfaces}

A connected compact complex surface $X$ is said to be a \textit{class $\mathrm{VII}$ surface} if $\kappa(X) = -\infty$ and $b_1(X) = 1$; furthermore, if $X$ is not the blowup of another complex surface (i.e. $X$ is \textit{minimal}), then it is said to be a \textit{class $\mathrm{VII}_0$ surface}.

The following theorem regarding class $\mathrm{VII}_0$ surfaces with $b_2 = 0$ first appeared in work by Bogomolov \cite{Bogomolov1}, \cite{Bogomolov2}. It was later reproved by Teleman \cite{Teleman0}, also see work by Li, Yau, and Zheng \cite{LYZ}.

\begin{theorem}
\label{b_2=0}
A class $\mathrm{VII}_0$ surface with $b_2 = 0$ is biholomorphic to a Hopf surface or an Inoue surface.
\end{theorem}

In the following sections, we review the constructions of Hopf and Inoue surfaces, before moving on to class $\mathrm{VII}_0$ surfaces with $b_2 > 0$.

\subsubsection{Hopf Surfaces}

A \textit{Hopf manifold} is a compact complex manifold whose universal cover is biholomorphic to $\mathbb{C}^n\setminus\{0\}$. A Hopf manifold with fundamental group isomorphic to $\mathbb{Z}$ is called a \textit{primary Hopf manifold}, otherwise it is called a \textit{secondary Hopf manifold}. Primary Hopf manifolds are sometimes defined to be quotients of $\mathbb{C}^n\setminus\{0\}$ by the infinite cyclic group generated by a contraction; Kodaira showed that these two definitions coincide. Moreover, he showed that every secondary Hopf manifold is finitely covered by a primary Hopf manifold, see \cite[Theorem 30]{Kodaira}\footnote{Kodaira only proves these statements in the case of surfaces, but the proofs apply for complex manifolds of arbitrary dimension.}. 

The structure of Hopf surfaces is well understood -- this is in part due to a normal form for contractions on $\mathbb{C}^2$. Up to an automorphism, every contraction on $\mathbb{C}^2$ is of the form $(z_1, z_2) \mapsto (\alpha_1z_1 + \lambda z_2^m, \alpha_2z_2)$ where $m$ is a positive integer and $\alpha_1, \alpha_2, \lambda \in \mathbb{C}$ are subject to the conditions $(\alpha_1 - \alpha_2^m)\lambda = 0$ and $0 < |\alpha_1| \leq |\alpha_2| < 1$; see \cite{Sternberg} and \cite{Lattes}. It follows that all primary Hopf surfaces are deformation equivalent, and hence diffeomorphic; namely, there are all diffeomorphic to $S^1\times S^3$. The diffeomorphism types of secondary Hopf surfaces were determined by Kato, see \cite[Theorem 9]{Kato1} and \cite{Kato2}:

\begin{theorem}
\label{SecHopf}
Every secondary Hopf surface is diffeomorphic to
\begin{enumerate}
\item[(1)] a product of $S^1$ and a quotient of $S^3$, or 
\item[(2)] the mapping torus of a quotient of $S^3$ by a diffeomorphism of order two or three.
\end{enumerate}
\end{theorem}

Of course, one could view the first family as mapping tori of a diffeomorphism of order one, namely the identity map. In particular, a secondary Hopf surface in the second family is double or triple covered by a secondary Hopf surface in the first family.

\subsubsection{Inoue Surfaces}

Inoue surfaces were introduced by Inoue in \cite{Inoue}, and further explained in \cite{Inouepm}. We outline the construction of the four families of Inoue surfaces $S^+_M$, $S^-_M$, $S^+_{N, p, q, r, t}$, and $S^-_{N, p, q, r}$, all of which have universal cover biholomorphic to $\mathbb{H}\times\mathbb{C}$. As Inoue surfaces are the main focus of this paper, we take this opportunity to give a relatively detailed description of the underlying smooth topology.

Let $M = [m_{ij}] \in SL(3, \mathbb{Z})$ be a matrix with one real eigenvalue $\alpha > 0$ and two complex eigenvalues $\beta$, $\overline{\beta}$ where $\operatorname{Im}(\beta) > 0$. Let $(a_1, a_2, a_3)$, $(b_1, b_2, b_3)$, and $(\overline{b_1}, \overline{b_2}, \overline{b_3})$ be eigenvectors of $M$ corresponding to the eigenvalues $\alpha$, $\beta$, and $\overline{\beta}$ respectively. By replacing $(a_1, a_2, a_3)$ with its real part if necessary, we can assume the entries are real.

Consider the following biholomorphisms of $\mathbb{H}\times\mathbb{C}$:
\begin{align*}
g_0(w, z) &= (\alpha w, \beta z)\\
g_i(w, z) &= (w + a_i, z + b_i),\quad i = 1, 2, 3.
\end{align*}
The group generated by $g_0, g_1, g_2, g_3$, which we denote by $G_M^+$, acts freely, properly discontinuously, and cocompactly on $\mathbb{H}\times\mathbb{C}$. The resulting quotient $S_M^+ := (\mathbb{H}\times\mathbb{C})/G_M^+$ is a compact complex surface. In order to better understand the topology of $S_M^+$, it will be useful to consider the subgroup of $G_M^+$ generated by $g_1, g_2, g_3$, which we denote by $\Gamma_M^+$; as $g_1, g_2, g_3$ commute, we see that $\Gamma_M^+ \cong \mathbb{Z}^3$.

Note that the action of $\Gamma_M^+$ on $(w, z)$ preserves $\operatorname{Im}(w)$ and the quotient of $\{(w, z) \in \mathbb{H}\times\mathbb{C} \mid \operatorname{Im}(w) = d\}$ by $\Gamma_M^+$ can be identified with $(\mathbb{R}\times\mathbb{C})/\Gamma_M^+$ by forgetting $\operatorname{Im}(w)$. The vectors $(a_1, b_1), (a_2, b_2), (a_3, b_3) \in \mathbb{R}\times\mathbb{C}$ are linearly independent over $\mathbb{R}$, and hence generate a lattice. Therefore $(\mathbb{R}\times\mathbb{C})/\Gamma_M^+$ is diffeomorphic to a three-torus $T^3$and $(\mathbb{H}\times\mathbb{C})/\Gamma_M^+$ is diffeomorphic to $(0, \infty)\times T^3$.

As $(a_1, a_2, a_3)$ and $(b_1, b_2, b_3)$ are eigenvectors of $M$ for the eigenvalues $\alpha$ and $\beta$ respectively, a simple calculation shows that $g_0g_ig_0^{-1} = g_1^{m_{i1}}g_2^{m_{i2}}g_3^{m_{i3}}$ for $i = 1, 2, 3$. So $g_0$ is in the normaliser of $\Gamma_M^+$ and therefore descends to a diffeomorphism $(\mathbb{H}\times\mathbb{C})/\Gamma_M^+ \to (\mathbb{H}\times\mathbb{C})/\Gamma_M^+$. Under the identification with $(0, \infty)\times T^3$, the map induced by $g_0$ restricts to a diffeomorphism between $\{1\}\times T^3$ and $\{\alpha\}\times T^3$. So $S_M^+ = (\mathbb{H}\times\mathbb{C})/G_M^+ = ((\mathbb{H}\times\mathbb{C})/\Gamma_M^+)/\langle g_0\rangle$ is diffeomorphic to a mapping torus of a self-diffeomorphism of $T^3$. By tracing through the identifications, it is not hard to show that the self-diffeomorphism is just the map on $T^3$ induced by $M^T$, see \cite[Proposition 2.9]{EP} for example.

The Inoue surface $S_M^-$ is defined similarly to $S_M^+$. Consider the following biholomorphisms of $\mathbb{H}\times\mathbb{C}$:
\begin{align*}
g_0(w, z) &= (\alpha w, \overline{\beta} z)\\
g_i(w, z) &= (w + a_i, z + \overline{b_i}),\quad i = 1, 2, 3.
\end{align*}
Let $G^-_M$ be the group generated by $g_0, g_1, g_2, g_3$ and define $S^-_M := (\mathbb{H}\times\mathbb{C})/G^-_M$. The arguments above concerning $S^+_M$ can be made analagously for $S^-_M$. More directly, note that the map $\mathbb{H}\times\mathbb{C} \to \mathbb{H}\times\mathbb{C}$ given by $(w, z) \mapsto (w, \overline{z})$ descends to a diffeomorphism $S^+_M \to S^-_M$.

\begin{remark}
There is some confusion in the literature regarding the Inoue surfaces $S^+_M$ and $S^-_M$. For a matrix $M \in SL(3, \mathbb{Z})$ with one real eigenvalue $\alpha > 1$ and two complex conjugate eigenvalues $\beta$, $\overline{\beta}$, Inoue defined in \cite{Inoue} a complex surface $S_M$ in the same way as $S^+_M$ was defined above. However, Inoue did not indicate how to distinguish between $\beta$ and $\overline{\beta}$; note that we required $\operatorname{Im}(\beta) > 0$. So, depending on the naming of the eigenvalues, the surface $S_M$ could be $S^+_M$ or $S^-_M$. In \cite{Inouepm}, Inoue separated the two cases and showed that they are not biholomorphic, or even deformation equivalent. In the literature which followed, the distinction between $S^+_M$ and $S^-_M$ was not always observed, furthermore, the notation $S^0$ is sometimes used to indicate the Inoue surfaces which Inoue denoted by $S_M$.
\end{remark}

Now let $N = [n_{ij}] \in SL(2, \mathbb{Z})$ be a matrix with two real eigenvalues $\alpha$ and $\frac{1}{\alpha}$ where $\alpha > 1$; this is equivalent to $\operatorname{tr}(N) > 2$. Let $(a_1, a_2)$ and $(b_1, b_2)$ be real eigenvectors of $N$ for $\alpha$ and $\frac{1}{\alpha}$ respectively. Let $p, q, r \in \mathbb{Z}$ with $r \neq 0$ and $t \in \mathbb{C}$. Consider the following equation for $c_1$ and $c_2$: 
\begin{equation}
\label{c_i+}
(I - N)\begin{bmatrix}c_1\\ c_2\end{bmatrix} = \begin{bmatrix}e_1\\ e_2\end{bmatrix} + \frac{1}{r}(a_2b_1 - a_1b_2)\begin{bmatrix}p\\ q\end{bmatrix}
\end{equation}
where $e_i = \frac{1}{2}n_{i1}(n_{i1} - 1)a_1b_1 + \frac{1}{2}n_{i2}(n_{i2} - 1)a_2b_2 + n_{i1}n_{i2}a_2b_1$; the reason for the seemingly peculiar form of equation (\ref{c_i+}) will become clear later. As $1$ is not an eigenvalue of $N$, the matrix $I - N$ is invertible, so there is a unique solution to the above equation. Moreover, $c_1$ and $c_2$ are real.

Consider the following biholomorphisms of $\mathbb{H}\times\mathbb{C}$:
\begin{align*}
g_0(w, z) &= (\alpha w, z + t)\\
g_i(w, z) &= (w + a_i, z + b_iw + c_i),\quad i = 1, 2\\
g_3(w, z) &= \left(w, z + \frac{1}{r}(a_2b_1 - a_1b_2)\right).
\end{align*}
The group generated by $g_0, g_1, g_2, g_3$, which we denote by $G^+_{N,p,q,r,t}$, acts freely, properly discontinuously, and cocompactly on $\mathbb{H}\times\mathbb{C}$. The resulting quotient $S_{N,p,q,r,t}^+ := (\mathbb{H}\times\mathbb{C})/G_{N,p,q,r,t}^+$ is a compact complex surface. In order to better understand the topology of $S_{N,p,q,r,t}^+$, it will be useful to consider the subgroup of $G_{N,p,q,r,t}^+$ generated by $g_1, g_2, g_3$, which we denote by $\Gamma_N^+$ for brevity.

Note that $\Gamma_N^+$ acting on $(w, z)$ preserves $\operatorname{Im}(w)$ and the quotient of $\{(w, z) \in \mathbb{H}\times\mathbb{C} \mid \operatorname{Im}(w) = d\}$ by $\Gamma^+_N$ can be identified with $(\mathbb{R}\times\mathbb{C})/\Gamma_N^+$ by forgetting $\operatorname{Im}(w)$. There is a diffeomorphism from $\mathbb{R}\times\mathbb{C}$ to $H(3, \mathbb{R})$, the three-dimensional real Heisenberg group, given by 
$$(\operatorname{Re}(w), z) \mapsto \begin{bmatrix} 1 & \zeta & \gamma\\ 0 & 1 & \eta\\ 0 & 0 & 1\end{bmatrix}$$
where $(\zeta, \eta)$ are the coordinates of $(\operatorname{Re}(w), \operatorname{Im}(z)/\operatorname{Im}(w))$ in the basis $\{(a_1, b_1), (a_2, b_2)\}$, and $\gamma = \operatorname{Re}(z)/(a_2b_1 - a_1b_2)$. Under this diffeomorphism, the action of $g_1$, $g_2$, and $g_3$ correspond to left multiplication by the matrices
$$\begin{bmatrix} 1 & 1 & c_1'\\ 0 & 1 & 0\\ 0 & 0 & 1\end{bmatrix},\quad \begin{bmatrix} 1 & 0 & c_2'\\ 0 & 1 & 1\\ 0 & 0 & 1\end{bmatrix},\quad\ \text{and}\ \quad \begin{bmatrix} 1 & 0 & \frac{1}{r}\\ 0 & 1 & 0\\ 0 & 0 & 1\end{bmatrix}$$
which generate a lattice $\Lambda$ in $H(3, \mathbb{R})$; here $c_i' = c_i/(a_2b_1 - a_1b_2)$. In particular, the quotient $(\mathbb{R}\times\mathbb{C})/\Gamma_N^+$ is diffeomorphic to a nilmanifold. Mapping an element of $H(3, \mathbb{R})/\Lambda$ to $(\zeta, \eta) + \mathbb{Z}^2$ defines a submersion to $T^2 = \mathbb{R}^2/\mathbb{Z}^2$ which endows $H(3, \mathbb{R})/\Lambda$ with the structure of a circle bundle over $T^2$. 

As the commutators of the matrices above do not depend on $c_1'$ and $c_2'$, the lattice $\Lambda$ is isomorphic to $\Lambda_r$, the lattice generated by the above matrices with zeroes in place of $c_1'$ and $c_2'$. As $H(3, \mathbb{R})/\Lambda$ and $H(3, \mathbb{R})/\Lambda_r$ are compact nilmanifolds with isomorphic fundamental groups, they are diffeomorphic \cite[page 235]{Auslander}. The latter is a quotient of the Heisenberg manifold $H(3, \mathbb{R})/H(3, \mathbb{Z})$ by $\mathbb{Z}_{r} = \mathbb{Z}/r\mathbb{Z}$ where $\mathbb{Z}_r$ acts on the $\gamma$ entry by translation. The Heisenberg manifold is the total space of the Chern class 1 circle-bundle over $T^2$, with the same projection map as described above. The $\mathbb{Z}_{r}$ action preserves the fibers, so it follows that $H(3, \mathbb{R})/\Lambda_r$ is the total space of the Chern class $r$ circle-bundle over $T^2$.

From the above, we see that $(\mathbb{H}\times\mathbb{C})/\Gamma_N^+$ is diffeomorphic to $(0, \infty)\times F$ where $F$ is the total space of the Chern class $r$ circle-bundle over $T^2$. By direct computation, one can show that $g_0g_1g_0^{-1} = g_1^{n_{11}}g_2^{n_{12}}g_3^p$ and $g_0g_2g_0^{-1} = g_1^{n_{21}}g_2^{n_{22}}g_3^q$. In fact, these two relations are equivalent to the constants $c_1, c_2$ satisfying equation (\ref{c_i+}), which explains its peculiar form. In addition, the biholomorphism $g_0$ commutes with $g_3$, so $g_0$ is in the normaliser of $\Gamma_N^+$ and therefore descends to a diffeomorphism $(\mathbb{H}\times\mathbb{C})/\Gamma_N^+ \to (\mathbb{H}\times\mathbb{C})/\Gamma_N^+$. Under the identification with $(0, \infty)\times F$, the map induced by $g_0$ restricts to a diffeomorphism between $\{1\}\times F$ and $\{\alpha\}\times F$. So $S_{N,p,q,r,t}^+ = (\mathbb{H}\times\mathbb{C})/G_{N,p,q,r,t}^+ = ((\mathbb{H}\times\mathbb{C})/\Gamma_N^+)/\langle g_0\rangle$ is diffeomorphic to a mapping torus of a self-diffeomorphism of $F$.

Finally, let $N = [n_{ij}] \in GL(2, \mathbb{Z})$ be a matrix with determinant $-1$ and real eigenvalues $\alpha$ and $-\frac{1}{\alpha}$ such that $\alpha > 1$. Let $(a_1, a_2)$ and $(b_1, b_2)$ be real eigenvectors of $N$ for $\alpha$ and $-\frac{1}{\alpha}$ respectively. Let $p, q, r \in \mathbb{Z}$ with $r \neq 0$. Consider the following equation for $c_1$ and $c_2$:
\begin{equation}
\label{c_i-}
(-I-N)\begin{bmatrix} c_1\\ c_2\end{bmatrix} = \begin{bmatrix} e_1\\ e_2\end{bmatrix} + \frac{1}{r}(a_2b_1 - a_1b_2)\begin{bmatrix} p\\ q\end{bmatrix}
\end{equation}
where $e_i = \frac{1}{2}n_{i1}(n_{i1} - 1)a_1b_1 + \frac{1}{2}n_{i2}(n_{i2} - 1)a_2b_2 + n_{i1}n_{i2}a_2b_1$ as before. As $-1$ is not an eigenvalue of $N$, the matrix $-I - N$ is invertible, so there is a unique solution to the above equation. Moreover, $c_1$ and $c_2$ are real.

Consider	the following biholomorphisms of $\mathbb{H}\times\mathbb{C}$:
\begin{align*}
g_0(w, z) &= (\alpha w, -z)\\
g_i(w, z) &= (w + a_i, z + b_iw + c_i),\quad i = 1, 2\\
g_3(w, z) &= \left(w, z + \frac{1}{r}(a_2b_1 - a_1b_2)\right).
\end{align*}
The group generated by $g_0, g_1, g_2, g_3$, which we denote by $G_{N, p, q, r}^-$, acts freely, properly discontinuously, and cocompactly on $\mathbb{H}\times\mathbb{C}$. The resulting quotient $S^-_{N,p,q,r} := (\mathbb{H}\times\mathbb{C})/G_{N,p,q,r}^-$ is a compact complex surface. Letting $\Gamma_N^-$ denote the subgroup generated by $g_1, g_2, g_3$, then, as before, the quotient $(\mathbb{H}\times\mathbb{C})/\Gamma_N^-$ is diffeomorphic to $(0, \infty)\times F$ where $F$ denotes the total space of the Chern class $r$ circle-bundle over $T^2$. As in the previous case, we have $g_0g_1g_0^{-1} = g_1^{n_{11}}g_2^{n_{12}}g_3^p$ and $g_0g_2g_0^{-1} = g_1^{n_{21}}g_2^{n_{22}}g_3^q$, which is equivalent to $c_1, c_2$ satisfying (\ref{c_i-}), but now $g_0g_3g_0^{-1} = g_3^{-1}$. Still, the biholomorphism $g_0$ is in the normaliser of $\Gamma_N^-$ and therefore descends to a diffeomorphism $(\mathbb{H}\times\mathbb{C})/\Gamma_N^- \to (\mathbb{H}\times\mathbb{C})/\Gamma_N^-$. Under the identification with $(0, \infty)\times F$, the map induced by $g_0$ restricts to a diffeomorphism between $\{1\}\times F$ and $\{\alpha\}\times F$. So $S_{N,p,q,r}^- = (\mathbb{H}\times\mathbb{C})/G_{N,p,q,r}^- = ((\mathbb{H}\times\mathbb{C})/\Gamma_N^-)/\langle g_0\rangle$ is diffeomorphic to a mapping torus of a self-diffeomorphism of $F$. 

Note that $N^2 \in SL(2, \mathbb{Z})$ and has real eigenvalues $\alpha^2 > 1$ and $\frac{1}{\alpha^2}$ with real eigenvectors $(a_1, a_2)$ and $(b_1, b_2)$. So if one considers the group of biholomorphisms of $\mathbb{H}\times\mathbb{C}$ generated by $g_0^2, g_1, g_2, g_3$, they appear to be the same as the biholomorphisms which would generate $G_{N^2, p_1, q_1, r, 0}^+$ for some $p_1, q_1 \in \mathbb{Z}$. However, for that to be the case, one must check that the $c_1, c_2$ given by equation (\ref{c_i-}) satisfy equation (\ref{c_i+}) for the matrix $N^2$ and integers $p_1$, $q_1$. This can be achieved by expanding $(g_0^2)g_1(g_0^2)^{-1}$ and $(g_0)^2g_2(g_0^2)^{-1}$ carefully, which will dictate the values of $p_1, q_1$. Therefore, the surface $S_{N^2, p_1, q_1, r, 0}^+$ double covers $S_{N, p, q, r}^-$. If particular, if $S_{N, p, q, r}^-$ is diffeomorphic to the mapping torus of $f\colon F \to F$, then $S_{N^2,p_1, q_1,r,0}^+$ is diffeomorphic to the mapping torus of $f\circ f$.

%\begin{align*}
%p_1 &= -n_{11}n_{12}n_{21}n_{22}r - \frac{1}{2}n_{11}(n_{11}-1)n_{11}n_{12}r - \frac{1}{2}n_{12}(n_{12}-1)n_{21}n_{22}r + (n_{11}-1)p + n_{12}q\\
%q_1 &= -n_{12}^2n_{21}^2r - \frac{1}{2}n_{21}(n_{21}-1)n_{11}n_{12}r - \frac{1}{2}n_{22}(n_{22}-1)n_{21}n_{22}r + n_{21}p + (n_{22}-1)q
%\end{align*}

We now record the following fact which follows from the constructions above.

\begin{proposition}
\label{Inouemt}
Every Inoue surface is diffeomorphic to a mapping torus of a compact three-dimensional nilmanifold $F$.
\end{proposition}

For those Inoue surfaces of the form $S_M^+$ or $S_M^-$, the fiber is $F = T^3$, while for those of the form $S_{N,p,q,r,t}^+$ or $S_{N,p,q,r}^-$, the fiber $F$ is the total space of the Chern class $r$ circle-bundle over $T^2$. Hasegawa has shown that all Inoue surfaces themselves are diffeomorphic to solvmanifolds, see \cite[Theorem 1]{Hasegawa}. In general, a compact solvmanifold $S$ is the total space of a fiber bundle with base a torus of dimension $b_1(S)$ and fiber a nilmanifold.

\subsubsection{Class $\mathrm{VII}_0$ surfaces with $b_2 > 0$} 

A \textit{spherical shell} in an $n$-dimensional complex manifold $X$ is an open subset $U$ which is diffeomorphic to a connected open subset of $S^{2n-1}$ in $\mathbb{C}^n\setminus\{0\}$. If $X\setminus U$ is connected, then $U$ is called a \textit{global spherical shell}. 

While every complex manifold contains a spherical shell, the existence of a global spherical shell is quite restrictive. Kato \cite[Theorem 1]{KatoGSS} showed that, for $n \geq 2$, if a connected complex manifold contains a global spherical shell, then it is a deformation of a modification of a primary Hopf manifold at finitely many points. In the case of surfaces, every modification is a composition of blowups of points \cite[Theorem 5.7]{Laufer}. Therefore, a connected compact complex surface $X$ which contains a global spherical shell, often called a \textit{Kato surface}, is a deformation of a blownup primary Hopf surface. Note that such a surface has $\kappa(X) = -\infty$ and $b_1(X) = 1$, and is therefore a class $\mathrm{VII}$ surface; the following conjecture asserts a partial converse is true.

\begin{gssconjecture}
A class $\mathrm{VII}_0$ surface with $b_2 > 0$ contains a global spherical shell, i.e. is a Kato surface.
\end{gssconjecture}

Note that the assumption $b_2 > 0$ is necessary as Inoue surfaces do not contain global spherical shells (they do not have fundamental group isomorphic to $\mathbb{Z}$). Similarly, the minimality condition is necessary, otherwise blownup Inoue surfaces would provide counterexamples.

Theorem \ref{b_2=0} states that the only class $\mathrm{VII}_0$ surfaces with $b_2 = 0$ are Hopf surface and Inoue surfaces. The global spherical shell conjecture asserts that class $\mathrm{VII}_0$ surfaces with $b_2 > 0$ are deformations of blownup primary Hopf surfaces. One might expect to be able to construct counterexamples by considering deformations of secondary Hopf surfaces or Inoue surfaces. This has been ruled out by work of Dloussky. In \cite{Dloussky1}, he shows that if a class $\mathrm{VII}_0$ surface with $b_2 > 0$ admits a global spherical shell, then any finite quotient also admits a global spherical shell; this rules out counterexamples which are deformations of blownup secondary Hopf surfaces. While counterexamples arising from deformations of blownup Inoue surfaces are prohibited by the arguments in the proof of \cite[Theorem 1.13]{Dloussky2}, namely case 5. 

By analysing the moduli space of stable holomorphic bundles on class $\mathrm{VII}_0$ surfaces, Teleman has verified the global spherical shell conjecture for small values of $b_2$, namely $b_2 = 1$ and $b_2 = 2$, see \cite{Teleman1} and \cite{Teleman2} respectively. In \cite{Teleman3}, Teleman has announced the case $b_2 = 3$ with the “long and technical” details to follow in an upcoming paper.

If the global spherical shell conjecture were indeed true, then every class $\mathrm{VII}_0$ surface with $b_2 > 0$ would be diffeomorphic to $(S^1\times S^3)\#b_2\overline{\mathbb{CP}^2}$.

\subsection{Kodaira Surfaces} A connnected compact complex surface $X$ is called a \textit{primary Kodaira surface} if $K_X$ is trivial and $b_1(X) = 3$. A \textit{secondary Kodaira surface} is a surface which admits a primary Kodaira surface as a covering.

We recall the following elementary fact:

\begin{proposition}
\label{Kodairasymp}
Primary Kodaira surfaces are symplectic.
\end{proposition}
\begin{proof}
Let $X$ be a primary Kodaira surface. As $K_X$ is holomorphically trivial, there is a nowhere-zero holomorphic two-form $\alpha$; let $\omega = 2\operatorname{Re}(\alpha) = \alpha + \overline{\alpha}$. As $\alpha$ is holomorphic, $\overline{\partial}\alpha = 0$ while $\partial\alpha = 0$ for bidegree reasons, so $\alpha$ is closed and hence so is $\omega$. In local holomorphic coordinates $(U, (z^1, z^2))$, we have $\alpha|_U = f dz^1\wedge dz^2$ where $f\colon U \to \mathbb{C}$ is a nowhere-zero holomorphic function. So $\omega^2|_U = 2\alpha\wedge\overline{\alpha} = |f|^2 dz^1\wedge dz^2\wedge d\overline{z}^1\wedge d\overline{z}^2$ which is nowhere-zero, hence $\omega$ is a non-degenerate closed two-form, i.e. a symplectic form.
\end{proof}

The proof only uses the fact that Kodaira surfaces are compact complex surfaces with trivial canonical bundle, so applies equally well to tori and $K3$ surfaces. In higher dimensions, triviality of the canonical bundle does not imply the existence of a symplectic form. For example, the connected sum of $m \geq 2$ copies of $S^3\times S^3$ admits a complex structure with trivial canonical bundle \cite{LT}, but has $b_2 = 0$. Also note that secondary Kodaira surfaces are not symplectic as they have $b_2 = 0$.

\section{Yamabe Invariant}

Let $M$ be a closed smooth $n$-dimensional manifold. The Einstein-Hilbert functional on the space of Riemannian metrics on $M$ is given by $g \mapsto \mathcal{E}(g)$ where
$$\mathcal{E}(g) = \frac{\displaystyle\int_Ms_gd\mu_g}{\operatorname{Vol}(M, g)^{\frac{n-2}{n}}}$$
is the normalised total scalar curvature; here $s_g$ denotes the scalar curvature of $g$ and $d\mu_g$ is the Riemannian volume density. The exponent in the denominator is such that $\mathcal{E}(cg) = \mathcal{E}(g)$ for any positive constant $c$, and hence $\mathcal{E}(g) = \mathcal{E}(g_1)$ where $g_1$ is the unique unit-volume metric homothetic to $g$. 

Hilbert first considered the unnormalised total scalar curvature functional and showed that when $n \geq 3$, the critical points are precisely Ricci-flat metrics, see \cite[Proposition 4.17]{Besse}. In the normalised case, the critical points, again for $n \geq 3$, are instead Einstein metrics (the vanishing of the Ricci tensor is replaced with the vanishing of the trace-free Ricci tensor) see \cite[Theorem 4.21]{Besse}. When $n = 2$, both the normalised and unnormalised total scalar curvatures are independent of the metric and is given by $4\pi\chi(M)$ due to the Gauss-Bonnet Theorem, hence the need for the restriction $n \geq 3$ in the previous statements.

When restricted to a conformal class $\mathcal{C}$, the critical points of the Einstein-Hilbert functional are constant scalar curvature metrics \cite[Proposition 4.25]{Besse}. By H\"older's inequality, the restricted functional is bounded below, so we define the \textit{Yamabe constant} of $\mathcal{C}$ to be the real number $Y(M, \mathcal{C}) = \inf_{g \in \mathcal{C}}Y(M, \mathcal{C})$. Yamabe \cite{Yamabe} claimed to prove that this infimum is always realised and hence every conformal class contains a constant scalar curvature metric. Trudinger \cite{Trudinger} pointed out a critical flaw in Yamabe's argument which was later rectified by works by Aubin \cite{Aubin} and Schoen \cite{Schoen}. It follows that there is a unit-volume metric in $\mathcal{C}$ with constant scalar curvature equal to $Y(M, \mathcal{C})$; if $Y(M, \mathcal{C}) \leq 0$, then this is the only unit-volume constant scalar curvature metric in $\mathcal{C}$, but this need not be the case if $Y(M, \mathcal{C}) > 0$ as the conformal class $[g_{\text{round}}]$ on $S^n$ demonstrates.

In his work on the Yamabe problem, Aubin showed that $Y(M, \mathcal{C}) \leq Y(S^n, [g_{\text{round}}])$. We define $Y(M) = \sup_{\mathcal{C}}Y(M, \mathcal{C})$ to be the \textit{Yamabe invariant} of $M$; note that $Y(M) \leq Y(S^n)$. We say the Yamabe invariant is \textit{realised} if there is a conformal class $\mathcal{C}$ for which $Y(M, \mathcal{C}) = Y(M)$, and hence there is a unit-volume metric with constant scalar curvature $Y(M)$. Note that $Y(M)$ is positive if and only if $M$ admits a metric of positive scalar curvature. If $Y(M)$ is non-positive, then $Y(M)$ can be characterised as the supremum of the values of the unit-volume constant scalar curvature metrics on $M$.

\subsection{Yamabe Invariants of Complex Surfaces}

We now review what is known about the values of the Yamabe invariant for complex surfaces. 

In the K\"ahler case, LeBrun \cite{LeBrun2} proved the following:

\begin{theorem}
\label{LeBrun}
Let $M$ be a a connected compact complex surface which admits a K\"ahler metric. Then $Y(M)$ is positive if $\kappa(M) = -\infty$, zero if $\kappa(M) = 0, 1$, and negative if $\kappa(M) = 2$.
\end{theorem}

While the value of the Yamabe invariant itself is not enough to distinguish between complex surfaces of Kodaira dimensions $0$ and $1$, LeBrun showed in the same paper that the two cases can be distinguished by asking when the Yamabe invariant is realised.

\begin{theorem}
\label{realised}
Let $M$ be a connected compact complex surface which admits a K\"ahler metric with $Y(M) = 0$, i.e. $\kappa(M) = 0$ or $1$. Then $Y(M)$ is realised if and only if $\kappa(M) = 0$ and $M$ is minimal.
\end{theorem}

In addition, LeBrun computed the value of Yamabe invariant precisely in the general type case. If $M$ is a general type surface with minimal model $X$, then $Y(M) = Y(X) = -4\pi\sqrt{2c_1(X)^2}$. The Yamabe invariant in the positive case is harder to compute in general. The only known value with $\kappa(M) = -\infty$ is $Y(\mathbb{CP}^2) = 12\sqrt{2}\pi$, see \cite{LeBrun}. Also note that while the value of the Yamabe invariant is unchanged after blowing up a K\"ahler surface of Kodaira dimension 0, 1, or 2 (i.e. it is a bimeromorphic invariant in these cases) it is unknown whether this is true of K\"ahler surfaces of Kodaira dimension $-\infty$.

There is no analogue of Theorem \ref{LeBrun} in higher dimensions, which is why we have set our attention to complex dimension two. To see this, let $m \geq 3$, and consider a smooth degree $m + 3$ hypersurface $M$ of $\mathbb{CP}^{m+1}$; note that $M$ is of general type. As $M$ is a simply connected non-spin closed smooth manifold of real dimension $2m \geq 6$, it admits a metric of positive scalar curvature by \cite[Corollary C]{GL}, and hence $Y(M) > 0$ in contrast with the surface case. More generally, Petean \cite{Petean2} has shown that any simply connected manifold of real dimension at least five has non-negative Yamabe invariant.

In the non-K\"ahler setting, much less is known about the values of the Yamabe invariant. All of our knowledge is restricted to Hopf surfaces and their blowups. In \cite[Section 2]{Schoen2}, Schoen showed that $Y(S^1\times S^{n-1}) = Y(S^n)$ which can be explicitly calculated as the latter is realised by the round metric. In particular, primary Hopf surfaces are all diffeomorphic to $S^1\times S^3$ and we have $Y(S^1\times S^3) = Y(S^4) = 8\sqrt{6}\pi$. It follows from a result of Gursky and LeBrun \cite[Theorem B]{GuL} that the Yamabe invariant of a primary Hopf surface blownup at $k$ points, which is diffeomorphic to $M_k := (S^1\times S^3)\#k\overline{\mathbb{CP}^2}$, satisfies $12\sqrt{2}\pi \leq Y(M_k) \leq 4\pi\sqrt{2k + 16}$ for $k = 1, 2, 3$; it is also true for $k > 3$, but the upper bound is no better than the upper bound $Y(S^4)$. In particular, we see that $Y(M_k) \neq Y(S^1\times S^3)$ for $k =1, 2, 3$, so the Yamabe invariant is not a bimeromorphic invariant for class $\mathrm{VII}$ surfaces. The upper bound also applies to secondary Hopf surfaces blownup at $k > 0$ points. However, while those secondary Hopf surfaces in class (1) of Theorem \ref{SecHopf} and their blowups admit metrics of positive scalar curvature, and hence have positive Yamabe invariant, this is not clear for those secondary Hopf surfaces in class (2).

Note, all of the known results in the previous paragraph are consistent with what one would expect the non-K\"ahler analogue of LeBrun's Theorem would predict: surfaces with Kodaira dimension $-\infty$ have positive Yamabe invariant. Modulo the status of the global spherical shell conjecture, the only non-K\"ahler surfaces which could possibly violate this prediction are secondary Hopf surfaces in class (2), Inoue surfaces, and their blowups.

While complex surfaces are the focus of this article, the Yamabe invariant has been computed for many four-manifolds which do not admit a complex structure, see \cite{IL}, \cite{IMN}, \cite{Petean1}, \cite{PR}, \cite{Sung}, and \cite{Suv}.

\section{Inoue Surfaces, Kodaira Surfaces, and their blowups}

In this section we prove the main result: the Yamabe invariants of Inoue surfaces and their blowups are zero. We begin by first noting that work of Paternian and Petean implies that the Yamabe invariant of these manifolds is non-negative.

Recall that a \textit{$\mathcal{T}$-structure} on a closed smooth manifold $M$ is a finite open cover $\{U_1, \dots, U_N\}$ and a non-trivial torus action on each $U_i$ such that each intersection $U_{i_1}\cap\dots\cap U_{i_k}$ is invariant under the torus actions on $U_{i_1}, \dots, U_{i_k}$ and the torus actions commute. This notion was introduced by Cheeger and Gromov in \cite{CG} as a special case of a more general notion called an $\mathcal{F}$-structure, although they stated the definition using different terminology.

A result of Paternain and Petean \cite[Theorem 7.2]{PP1} asserts that, provided $\dim M > 2$, the existence of a $\mathcal{T}$-structure implies $Y(M) \geq 0$. They also showed that all Inoue surfaces admit a $\mathcal{T}$-structure \cite[Theorem B]{PP2}. The simplest case is those Inoue surfaces of type $S_M^+$ (or $S_M^-$ as they are diffeomorphic). Any such surface is the mapping torus of the self-diffeomorphism of $T^3$ induced by $M^T$ . Let $p\colon S_M^+ \to S^1$ denote the projection, then let $U_1 = p^{-1}(S^1\setminus\{1\})$ and $U_2 = p^{-1}(S^1\setminus\{-1\})$. Note that $U_1$ and $U_2$ are both diffeomorphic to $(0, 1) \times T^3$ so they admit effective torus actions acting by translations on the second factor. Moreover, the intersection $U_1\cap U_2$ is invariant under the torus actions, and as the diffeomorphism $T^3 \to T^3$ is linear, they commute and hence $X$ has a $\mathcal{T}$-structure.

As for the blowups of Inoue surfaces, note that $Y(\overline{\mathbb{CP}^2}) = 12\sqrt{2}\pi > 0$, so by a theorem of Kobayashi on the Yamabe invariant of a connected sum \cite[Theorem 2 (a)]{Kobayashi}, it follows that the blowups also have non-negative Yamabe invariant. Alternatively, as a non-trivial $S^1$-action is a $\mathcal{T}$-structure, we see that $\overline{\mathbb{CP}^2}$ admits a $\mathcal{T}$-structure, and hence so do the blowups of Inoue surfaces by \cite[Theorem 5.9]{PP1}. In summary, we have the following:

\begin{proposition}
\label{Yamabenn}
Let $M$ be an Inoue surface or a blownup Inoue surface. Then $Y(M) \geq 0$.
\end{proposition}

So all that remains to be shown is that these surfaces do not admit metrics of positive scalar curvature. As all Inoue surfaces are mapping tori by Proposition \ref{Inouemt}, we first note the following elementary fact:

\begin{proposition}
\label{fibre}
Let $F$ be a closed orientable $n$-dimensional manifold, and let $p\colon M_f \to S^1$ be the mapping torus of some orientation-preserving homeomorphism $f\colon F \to F$. The inclusion $F \hookrightarrow M_f$ induces an injection $H_n(F; \mathbb{Z}) \to H_n(M_f; \mathbb{Z})$. Moreover, the image of $[F]$ is primitive.
\end{proposition}
\begin{proof}
Letting $U = p^{-1}(S^1\setminus\{i\})$ and $V = p^{-1}(S^1\setminus\{-i\})$, we have inclusion maps $k^U\colon U\cap V \to U$, $k^V\colon U\cap V \to V$, $\ell^U\colon U \to M_f$, and $\ell^V\colon V \to M_f$. Identifying $F$ with $p^{-1}(1)$, we also have inclusion maps $j : F \to U$ and $\iota : F \to M_f$. Now consider the following diagram (all homology groups are with $\mathbb{Z}$ coefficients)
$$\begin{tikzcd}[row sep = large, column sep = large]
H_{n+1}(M_f) \arrow{r} & H_n(U\cap V) \arrow{r}{(k^U_*, k^V_*)} & H_n(U)\oplus H_n(V) \arrow{r}{\ell^U_* - \ell^V_*} & H_n(M_f) \arrow{r} & H_{n-1}(U\cap V)\\
 & & H_n(F) \arrow{u}{(j_*, 0)} \arrow{ur}[swap]{\iota_*} & & 
\end{tikzcd}$$
The first row is exact by Mayer-Vietoris, and as $(\ell^U_* - \ell^V_*)\circ(j_*, 0) = \ell^U_*\circ j_* - \ell^V_*\circ 0= (\ell^U\circ j)_* = \iota_*$, the triangle commutes.

Note that $H_{n+1}(M_f) \cong \mathbb{Z}$, $H_n(U\cap V) \cong H_n(F\sqcup F) \cong \mathbb{Z}^2$, and $H_n(U) \cong H_n(V) \cong H_n(F) \cong \mathbb{Z}$. As each connected component of $U\cap V$ includes into both $U$ and $V$, the map $\mathbb{Z}^2 \to \mathbb{Z}^2$ is multiplication by $\begin{bsmallmatrix} 1 & 1\\ 1 & 1\end{bsmallmatrix}$. The kernel of this map is the span of $\begin{bsmallmatrix}1\\ -1\end{bsmallmatrix}$, so by exactness, the map $\mathbb{Z} \to \mathbb{Z}^2$ is multiplication by $\begin{bsmallmatrix}1\\ -1\end{bsmallmatrix}$. Furthermore, $H_n(F) \cong \mathbb{Z}$ and the map $(j_*, 0)$ corresponds to multiplication by $\begin{bsmallmatrix}1\\ 0\end{bsmallmatrix}$. As $\begin{bsmallmatrix}1\\ 0\end{bsmallmatrix}$ is not in the image of the map $\mathbb{Z}^2 \to \mathbb{Z}^2$, it follows that $\iota_*$ is injective and has the same image as $\ell^U_* - \ell^V_*$.

Suppose now that $\iota_*[F] = a\gamma$ for some $\gamma \in H_n(M_f)$. Then the image of $\gamma$ under $H_n(M_f) \to H_{n-1}(U\cap V)$ is $a$-torsion, but $H_{n-1}(U\cap V) \cong H_{n-1}(F\sqcup F)$ is torsion-free, so the image of $\gamma$ is zero. By exactness, $\gamma$ is in the image of $\ell_U^* - \ell_V^*$ and hence in the image of $\iota_*$. As $[F]$ generates $H_n(F)$, we see that $\gamma = \iota_*(b[F])$ for some $b \in \mathbb{Z}$, but then $\iota_*[F] = a\gamma = \iota_*(ab[F])$. It follows from the injectivity of $\iota_*$ that $a = \pm 1$, and hence $\iota_*[F]$ is primitive.
\end{proof}

In the case that $b_n(M_f) = 1$, as is the case for Inoue surfaces, we obtain the following corollary.

\begin{corollary}
\label{generates}
Let $F$ be a closed orientable $n$-dimensional manifold, and let $p\colon M_f \to S^1$ be the mapping torus of some orientation-preserving homeomorphism $f\colon F \to F$. Suppose that $b_n(M_f) = 1$. Then $H_n(M_f; \mathbb{Z})$ is generated by the image of $[F]$ under the map $H_n(F; \mathbb{Z}) \to H_n(M_f; \mathbb{Z})$ induced by the inclusion $F \hookrightarrow M_f$.
\end{corollary}

We will need the following lemma:

\begin{lemma}
\label{connected}
Let $M$ be a closed connected smooth manifold with a closed connected smooth hypersurface $\Sigma$. If $\Sigma$ is non-orientable, or $[\Sigma] \in H_{n-1}(M; \mathbb{Z})$ is non-zero, then $M\setminus\Sigma$ is connected.
\end{lemma}
\begin{proof}
Let $U$ be a tubular neighbourhood of $\Sigma$ in $M$, and let $V = M\setminus\Sigma$. Then by Mayer-Vietoris, we have
$$\dots \to H_0(U\cap V) \to H_0(U)\oplus H_0(V) \to H_0(M) \to 0.$$
Note that $H_0(U) \cong H_0(M) \cong \mathbb{Z}$ as $\Sigma$ and $M$ are connected, and $H_0(V) \cong \mathbb{Z}^k$ where $k$ is the number of connected components of $V$. Also note that $U\cap V$ deformation retracts onto the orientation double cover of $\Sigma$, so if $\Sigma$ is non-orientable we have
$$\dots \to \mathbb{Z} \to \mathbb{Z}^{k+1} \to \mathbb{Z} \to 0$$
from which it immediately follows that $k = 1$, i.e. $V = M\setminus\Sigma$ is connected. If $\Sigma$ is orientable, we instead have
$$\dots \to \mathbb{Z}^2 \to \mathbb{Z}^{k+1} \to \mathbb{Z} \to 0$$
which implies that $k \leq 2$. Suppose that $k = 2$.

As $\Sigma$ and $M$ are orientable, the normal bundle is trivial, so there is a diffeomorphism $\phi\colon U \to (-1, 1)\times\Sigma$ such that $\phi|_{\Sigma}$ is the inverse of the inclusion $i\colon \Sigma \to M$. Let $M_-$ and $M_+$ be the connected components of $M\setminus\Sigma$ where $\phi(M_-\cap U) = (-1, 0)\times\Sigma$ and $\phi(M_+\cap U) = (0, 1)\times\Sigma$. Let $f\colon (-1, 1) \to \mathbb{R}$ be a non-decreasing smooth function such that $f|_{\left(-1, -\frac{1}{2}\right)} \equiv 0$ and $f|_{\left(\frac{1}{2}, 1\right)} \equiv 1$. Let $\hat{f}\colon M \to \mathbb{R}$ be the function given by $\hat{f}(p) = f(\operatorname{pr}_1(\phi(p)))$ for $p \in U$, and extended by $0$ on $M_-\setminus U$ and $1$ on $M_+\setminus U$. Then $d\hat{f}$ is a smooth one-form such that $d\hat{f}|_{M\setminus U} = 0$ and $(\phi^{-1})^*(d\hat{f}|_U) = g(t)dt$ where $g\colon (-1, 1) \to \mathbb{R}$ is a non-negative smooth function such that $g|_{\left(-1, -\frac{1}{2}\right)\cup\left(\frac{1}{2}, 1\right)} \equiv 0$ and $\int_{-1}^1g(t)dt = 1$. It follows that for any $(n - 1)$-form $\eta$ on $M$, we have $\int_M\eta\wedge d\hat{f} = \int_\Sigma i^*\eta$, so $d\tilde{f}$ is a representative of the Poincar\'e dual of $[\Sigma]$ under the map $\Phi\colon H^1(M; \mathbb{Z}) \to H^1_{\text{dR}}(M)$ given by change of coefficients followed by the de Rham isomorphism. As $H^1(M; \mathbb{Z})$ is torsion-free, the map $\Phi$ is injective, so $\operatorname{PD}([\Sigma]) = 0$ as $[d\hat{f}] = 0$ in $H^1_{\text{dR}}(M)$. Finally, as $H_{n-1}(M; \mathbb{Z})$ is torsion-free, we see that $[\Sigma] = 0$.
\end{proof}

We now briefly recall two methods which can be used to rule out the existence of positive scalar curvature metrics on a closed orientable smooth manifold $M$. 

\begin{enumerate} 
\item We say that $(M, g)$ is \textit{enlargeable} if for every $\varepsilon > 0$, there is an orientable Riemannian covering space which admits an $\varepsilon$-contracting map onto $(S^n, g_{\text{round}})$ which is constant at infinity and is of non-zero degree. As $M$ is compact, any two metrics are Lipschitz equivalent, so enlargeability is metric independent. The class of enlargeable manifolds is quite large, it includes manifolds of non-positive sectional curvature, nilmanifolds, and is closed under domination, i.e. if $M$ admits a map of non-zero degree onto an enlargeable manifold, then $M$ is enlargeable. Gromov and Lawson  proved that if $M$ is spin and enlargeable, then $M$ does not admit metrics of positive scalar curvature \cite[Theorem A]{GL1}.

\item Another approach, due to Schoen and Yau, is via stable minimal hypersurfaces. In the proof of \cite[Theorem 1]{SY79}, they showed that if $g$ is a positive scalar curvature metric on $M$, and $\Sigma$ is a closed stable minimal hypersurface of $(M, g)$, then $g|_{\Sigma}$ is conformal to a metric of positive scalar curvature. 
\end{enumerate}

We will combine these two approaches to prove the following:

\begin{theorem}
\label{Inouepsc}
Inoue surfaces and their blowups do not admit positive scalar curvature metrics.
\end{theorem}
\begin{proof}
Let $M$ denote the possibly blownup Inoue surface, and let $X$ denote the minimal model for $M$, so that $M$ is diffeomorphic to $X\# k\overline{\mathbb{CP}^2}$ for some $k\geq 0$. By Proposition \ref{Inouemt}, the manifold $X$ is a mapping torus of a three-dimensional nilmanifold $F$ (either a torus or a circle bundle over a two-dimensional torus).

Let $g$ be a positive scalar curvature metric on $M$. As $[F] \neq 0$ by Proposition \ref{fibre}, there are closed oriented stable minimal hypersurface $\Sigma_1, \dots, \Sigma_{\ell}$ and integers $m_1, \dots, m_{\ell}$ such that $m_1[\Sigma_1] + \dots + m_{\ell}[\Sigma_{\ell}] = [F]$, see \cite[Remark 3.4]{Lawson}. By construction, we have $[\Sigma_j] \neq 0$ for $j = 1, \dots, \ell$. Letting $\Sigma := \Sigma_1$, and noting that $b_3(M) = b_3(X) = 1$, we see that there is a non-zero integer $n$ such that $[\Sigma] = n[F]$ by Corollary \ref{generates}.

%As $\Sigma$ and $M$ are orientable, the normal bundle is trivial, so there is an embedding $\phi\colon (-1, 1)\times\Sigma \to M$ such that $\phi(0, \cdot)$ is the inclusion $i\colon \Sigma \hookrightarrow M$. Let $Y = M\setminus \phi((-\frac{1}{2}, \frac{1}{2})\times\Sigma)$. Note that $Y$ is a compact manifold with boundary $\partial Y = \Sigma\sqcup \Sigma$, and is connected by Lemma \ref{connected}.

As $\Sigma$ and $M$ are orientable, the normal bundle is trivial, so there is an embedding $\phi\colon (-1, 1)\times\Sigma \to M$ such that $\phi(0, \cdot)$ is the inclusion $i\colon \Sigma \hookrightarrow M$. Note that $M\setminus\Sigma$ is a connected non-compact manifold by Lemma \ref{connected}. We can compactify $M\setminus\Sigma$ to obtain a manifold $Y$ with boundary $\partial Y = \Sigma\sqcup\Sigma$ and interior diffeomorphic to $M\setminus\Sigma$. To see this, let $\phi_-\colon (-1, 0)\times\Sigma \to M\setminus\Sigma$ and $\phi_+\colon (0, 1)\times\Sigma \to M\setminus\Sigma$ be the embeddings given by restricting $\phi$ appropriately and set $Y = (\phi([0, 1)\times\Sigma) \sqcup (M\setminus\Sigma) \sqcup \phi((-1, 0]\times\Sigma))/\sim$ where $\phi_+(t, p) \sim \phi(t, p)$ for $(t, p) \in (0, 1)\times\Sigma$ and $\phi_-(t, p) \sim \phi(t, p)$ for $(t, p) \in (-1, 0)\times\Sigma$. 

Fix a diffeomorphism $f\colon \Sigma\times\{0,1\} \to \partial Y$. Now let $Z = (Y\times\mathbb{Z})/\sim$ where $(f(s, 1), m) \sim (f(s, 0), m+1)$ for all $s \in \Sigma$ and $m \in \mathbb{Z}$. Note that the self-map of $Y\times\mathbb{Z}$ given by $(y, m) \mapsto (y, m+1)$ descends to a fixed-point free diffeomorphism of $Z$. The quotient of $Z$ by the $\mathbb{Z}$-action generated by this diffeomorphism can be naturally identified with $M$, so we have a regular covering $\rho\colon Z \to M$ with group of deck transformations $\mathbb{Z}$. Hence, there is a short exact sequence
$$0 \to \pi_1(Z) \xrightarrow{\rho_*} \pi_1(M) \xrightarrow{\alpha} \mathbb{Z} \to 0.$$
Using the fact that $X$ is a mapping torus, we can construct another covering. First note that there is a regular covering $F\times\mathbb{R} \to X$ with deck transformation group $\mathbb{Z}$ generated by $(x, t) \mapsto (f(x), t + 1)$ where $f$ is the diffeomorphism which gives rise to the mapping torus $X$. There is a corresponding regular covering $\tau\colon M' \to M$ where $M'$ is diffeomorphic to the connected sum of $F\times\mathbb{R}$ with a copy of $k\overline{\mathbb{CP}^2}$ at $(p, m+\frac{1}{2})$ for some fixed $p \in F$ and all $m \in \mathbb{Z}$. Hence, there is a short exact sequence
$$0 \to \pi_1(M') \xrightarrow{\tau_*} \pi_1(M) \xrightarrow{\beta} \mathbb{Z} \to 0.$$
As $b_1(M) = b_1(X) = 1$, we see that $\alpha, \beta \in \operatorname{Hom}(\pi_1(M), \mathbb{Z}) \cong H^1(M; \mathbb{Z}) \cong \mathbb{Z}$, so there is $k \in \mathbb{Z}$ such that $\alpha = k\beta$ or $\beta = k\alpha$. As $\alpha$ and $\beta$ are surjective, it follows that $k = \pm 1$ and hence $\pi_1(Z) \cong \ker\alpha = \ker\beta \cong \pi_1(M')$. Therefore, the two coverings are isomorphic; let $\Psi\colon Z \to M'$ be an isomorphism of covering spaces.

%without loss of generality, suppose $\alpha = k\beta$. As $X$ is aspherical and $\pi_1(M) \cong \pi_1(X)$, the group $\pi_1(M)$ is torsion-free by Proposition \ref{asphericaltf}, so $\pi_1(Z) \cong \ker(\alpha) = \ker(k\beta) = \ker(\beta) \cong \pi_1(M')$. Therefore, the two coverings are isomorphic; let $\Psi\colon Z \to M'$ be an isomorphism of covering spaces.

The inclusion of $\Sigma$ as one of the boundary components of $Y$ induces an inclusion $\iota\colon \Sigma \hookrightarrow Z$ such that $\rho\circ\iota = i$, so we have the following commutative diagram:
$$\begin{tikzcd}
\Sigma \arrow{r}{\iota} \arrow[swap]{drr}{i} & Z  \arrow{dr}{\rho} \arrow{r}{\Psi} & M' \arrow{d}{\tau}\\
 &  & M 
\end{tikzcd}$$
As $i_*[\Sigma] = n[F] \neq 0$, we see that $(\Psi\circ\iota)_*[\Sigma] \neq 0$. In fact, as $\tau_*$ is an isomorphism on $H_3$, we see that $(\Psi\circ\iota)_*[\Sigma] = n[F]$. Note that there is a map $g\colon M' \to F$ by first mapping to $F\times\mathbb{R}$ then projecting onto $F$, and $g_*$ is an isomorphism on $H_3$. Therefore $g\circ\Psi\circ\iota\colon \Sigma \to F$ is a map of three-manifolds with $(g\circ\Psi\circ\iota)_*[\Sigma] = n[F]$; that is, the map has degree $n \neq 0$. As $F$ is a nilmanifold, it is enlargeable, and therefore $\Sigma$ is also enlargeable. Moreover, as $\Sigma$ is a closed orientable three-manifold, it is spin, and hence does not admit a metric of positive scalar curvature. This is a contradiction as $\Sigma$ is a stable minimal hypersurface and $g$ is a metric of positive scalar curvature.
\end{proof}

The previous argument proves the following statement immediately:

\begin{theorem}
Let $X$ be the mapping torus of a diffeomorphism $f\colon F \to F$ where $F$ is a closed enlargeable three-manifold. If $b_1(X) = 1$, then $X\# N$ does not admit a positive scalar curvature metric for any closed orientable four-manifold $N$ with $b_1(N) = 0$.
\end{theorem}

\begin{remark}
It follows from Proposition IV.6.10 and Corollary IV.6.12 of \cite{LM} that $X\# N$ is weakly enlargeable, a concept which we will not define here. Every compact enlargeable manifold is weakly enlargeable, but it is not known if the converse is true. It is also not known if a non-spin weakly enlargeable manifold can admit metrics of positive scalar curvature.
\end{remark}

\begin{theorem}
Let $M$ be an Inoue surface or a blownup Inoue surface. Then $Y(M) = 0$, but it is not realised.
\end{theorem}
\begin{proof}
Combining Proposition \ref{Yamabenn} and Theorem \ref{Inouepsc} gives $Y(M) = 0$. If the Yamabe invariant were realised, then $M$ would admit a scalar flat metric, but any such metric must be Ricci flat as $M$ does not admit positive scalar metrics \cite[Lemma 5.2]{KW}. By \cite[Theorem A]{LeBrun3}, non-K\"ahler surfaces cannot admit Ricci flat metrics.
\end{proof}

So we see that the K\"ahler hypothesis of Theorem \ref{LeBrun} is necessary.

We take this opportunity to record the following result, which was alluded to by LeBrun in \cite{LeBrun2}, and also follows from the extension of LeBrun's Theorem to the symplectic setting in \cite{S-S&T}:

\begin{theorem}
\label{Kodairapsc}
Let $M$ be a Kodaira surface or a blownup Kodaira surface. Then $Y(M) = 0$ and it is not realised.
\end{theorem}
\begin{proof}
As in the Inoue surface case, Paternain and Petean showed that Kodaira surfaces admit $\mathcal{T}$-structures \cite[Theorem B]{PP2}, and hence Kodaira surfaces and their blowups have non-negative Yamabe invariant. 

Now note that by Proposition \ref{Kodairasymp}, primary Kodaira surfaces are symplectic. The complex blowup of a complex manifold is diffeomorphic to the symplectic blowup, so blownup primary Kodaira surfaces are symplectic. Such surfaces therefore have a non-trivial Seiberg-Witten invariant \cite{Taubes}, and as $b^+ = 2h^{2,0} = 2 > 1$, they cannot admit metrics of positive scalar curvature \cite[Corollary 2.3.8]{Nicolaescu}. The same is true of their free quotients, namely secondary Kodaira surfaces and their blowups.

Combining, we see that the Yamabe invariant of Inoue surfaces and their blowups is zero. Again by \cite[Theorem A]{LeBrun3}, non-K\"ahler surfaces cannot admit Ricci flat metrics and hence the Yamabe invariant is not realised.
\end{proof}

So we see that the K\"ahler hypothesis of Theorem \ref{realised} is also necessary.

Modulo the status of the global spherical shell conjecture, the only connected compact non-K\"ahler surfaces for which the Yamabe invariant is unknown are the secondary Hopf surfaces in class (2) of Theorem \ref{SecHopf}, properly elliptic non-K\"ahler surfaces, and their blowups. The author plans to address these remaining cases in a future paper.

\end{document}